\documentclass[11pt]{amsart}

\usepackage[a4paper,hmargin=3.5cm,vmargin=4cm]{geometry}
\usepackage{amsfonts,amssymb,amscd,amstext}
\usepackage{graphicx,float}
\usepackage[dvips]{epsfig}
\usepackage{pstricks,pst-plot}
\usepackage{hyperref}

\usepackage{fancyhdr}
\usepackage{enumerate}
\usepackage{titlesec}
\usepackage{mathrsfs}

\titleformat{\section}
{\filcenter\bfseries\large} {\thesection{.}}{0.2cm}{}
\titleformat{\subsection}[runin]
{\bfseries} {\thesubsection{.}}{0.15cm}{}[.]
\titleformat{\subsubsection}[runin]
{\em}{\thesubsubsection{.}}{0.15cm}{}[.]

\usepackage[up,bf]{caption}

\newtheorem{theorem}{Theorem}[section]
\newtheorem{proposition}[theorem]{Proposition}
\newtheorem{claim}[theorem]{Claim}
\newtheorem{lemma}[theorem]{Lemma}
\newtheorem{corollary}[theorem]{Corollary}

\theoremstyle{definition}
\newtheorem{definition}[theorem]{Definition}
\newtheorem{remark}[theorem]{Remark}

\numberwithin{equation}{section}
\numberwithin{figure}{section}

\newcommand\Hcal{\mathcal{H}}

\newcommand\Mcal{\mathcal{M}}
\newcommand\Ncal{\mathcal{N}}

\newcommand\Pcal{\mathcal{P}}

\newcommand\Ascr{\mathscr{A}}

\newcommand\Cscr{\mathscr{C}}

\newcommand\Oscr{\mathscr{O}}

\def\c{\mathbb{C}}

\newcommand\cp{\mathbb{CP}}

\def\n{\mathbb{N}}
\renewcommand\r{\mathbb{R}}
\newcommand\s{\mathbb{S}}

\newcommand\z{\mathbb{Z}}
\renewcommand\k{\mathbb{K}}

\newcommand\igot{\mathfrak{i}}

\renewcommand\igot{\mathfrak{i}}

\newcommand\pgot{\mathfrak{p}}

\newcommand\Agot{\mathfrak{A}}
\newcommand\Hgot{\mathfrak{H}}

\renewcommand\imath{\igot}

\newcommand\wt{\widetilde}
\newcommand\wh{\widehat}
\newcommand\di{\partial}

\newcommand\dist{\mathrm{dist}}

\newcommand\length{\mathrm{length}}

\newcommand\Flux{\mathrm{Flux}}

\usepackage{color}

\begin{document}

\fancyhead[LO]{Interpolation by complete minimal surfaces whose Gauss map misses two points} 
\fancyhead[RE]{I.\ Castro-Infantes}
\fancyhead[RO,LE]{\thepage}

\thispagestyle{empty}

\vspace*{6mm}
\begin{center}
{\bf \LARGE Interpolation by complete minimal surfaces whose Gauss map misses two points}

\vspace*{5mm}

{\large\bf Ildefonso Castro-Infantes}
\end{center}

\vspace*{7mm}

\begin{quote}
{\small
\noindent {\bf Abstract} \hspace*{0.1cm}
Let $M$ be an open Riemann surface and let $\Lambda\subset M$ be a closed discrete subset.
In this paper, we prove the existence of complete conformal minimal immersions $M\to\mathbb{R}^n$, $n\ge 3$, with prescribed values on $\Lambda$ and whose generalized Gauss map $M\to\mathbb{CP}^{n-1}$, $n\ge 3$, avoids $n$ hyperplanes of $\mathbb{CP}^{n-1}$ located in general position. In case $n=3$, we obtain complete nonflat conformal minimal immersions whose Gauss map $M\to\mathbb{S}^2$ omits two (antipodal) values of the sphere.

This result is deduced as a consequence of an interpolation theorem for conformal minimal immersions $M\to\mathbb{R}^n$ into the Euclidean space $\mathbb{R}^n$, $n\ge 3$, with $n-2$ prescribed components. 

\medskip

\noindent{\bf Keywords} \hspace*{0.1cm} 
minimal surface, Riemann surface, interpolation theory, harmonic map, Gauss map.

\medskip

\noindent{\bf MSC (2010)} \hspace*{0.1cm} 
53A10, 
53C42, 
32E30, 
30F15.  	
}
\end{quote}


\section{Introduction}\label{sec:intro}

%
Conformal minimal immersions of an open Riemann surface with values in a Euclidean space $\r^n$, $n\ge 3$, are harmonic maps. This fact has greatly influenced the study of minimal surfaces providing powerful techniques coming from complex analysis. One of the most interesting result is Mergelyan theorem  \cite{Runge1885AM,Mergelyan1951DAN,Bishop1958PJM}, which has allowed, among other achievements, the development of a theory of uniform approximation on compact subset \cite{AlarconLopez2012JDG,AlarconForstnericLopez2016MZ}, tangential approximation on some unbounded subsets \cite{CastroBrett2019}, and a theory of interpolation on discrete subsets \cite{AlarconCastro2017,AlarconCastroLopez2017,AlarconLopez2019Mittag} for conformal minimal immersions into the Euclidean spaces.

Let $M$ be an open Riemann surface and let $\partial$ denote the complex linear part of the exterior differential, $d=\di+\overline\di$ on $M$ where $\overline\di$ denotes the antilinear part. Given a conformal minimal immersion $X\colon M\to\r^n$, $\di X$ determines the Kodaira type holomorphic map
\begin{equation}\label{eq:ggm}
G_X\colon M\to\cp^{n-1}, M\ni p\mapsto G_X(p)=[\partial X(p)]
\end{equation}
which assumes values in the complex hyperquadric
\[
Q_{n-2}:=\{[z_1,\ldots,z_n]\in \cp^{n-1}: z_1^2+\ldots +z_n^2=0 \}\subset \cp^{n-1}
\]
and is known as the {\em generalized Gauss map of $X$}. Conversely, every holomorphic map $M\to Q_{n-2}$ is the generalized Gauss map of a conformal minimal immersion $M\to\r^n$, see \cite{AlarconForstnericLopez2016Gauss}. The study of the (generalized) Gauss map of a conformal minimal immersion is an important topic in the theory of minimal surfaces. Chern and Osserman \cite{Chern1965-book,ChernOsserman1967JAM} proved that if $X$ is complete then $X(M)$ is a plane or $G_X(M)$ intersects a dense subset of complex hyperplanes. Moreover, Ru \cite{Ru1991JDG} showed that the Gauss map of a complete nonflat conformal minimal immersion omits at most $n(n+1)/2$ hyperplanes in $\cp^{n-1}$ located in general position, generalizing a result of Fujimoto \cite{Fujimoto1990JDG}. In addition, this upper bound is sharp for some values of $n\ge 3$, see Fujimoto \cite{Fujimoto1988SRKU}.
Ahlfors \cite{Ahlfors1941ASSF} proved that if $G\colon \c\to\cp^{n-1}$ is a holomorphic map avoiding $n+1$ hyperplanes in general position, then $G$ is a {\em degenerate} map, that is, its image lies in a hyperplane of $\cp^{n-1}$. Concerning this, Alarc\'on, Fern\'andez, and L\'opez \cite{AlarconFernandezLopez2013CVPDE} proved that for any open Riemann surface $M$, there exists a complete conformal minimal immersion $X\colon M\to\r^n$ whose generalized Gauss map $G_X$ is nondegenerate and fails to intersect $n$ hyperplanes in general position; the number $n$ here is the maximum possible by Ahlfors theorem.
We show in this paper that one can prescribe the values of such an immersion on a closed discrete subset of $M$. 
\begin{theorem}\label{th:GXsimple}
	Let $M$ be an open Riemann surface and $\Lambda\subset M$ be a closed discrete subset. Any map $\Lambda\to\r^n$, $n\ge 3$,  extends to a complete conformal minimal immersion $X\colon M\to\r^n$ whose generalized Gauss map $G_X\colon M\to\cp^{n-1}$ is nondegenerate and fails to intersect $n$ hyperplanes of $\cp^{n-1}$ located in general position.
\end{theorem}
\noindent Note that the assumptions on $\Lambda$ are necessary since, by the Identity Principle, it is not possible to prescribe the values of a conformal minimal immersion $M\to\r^n$ on a subset that has an accumulation point.
In case $n=3$ we obtain the following.
\begin{corollary}\label{cor:gaussmap3}
	Let $M$ be an open Riemann surface and $\Lambda\subset M$ be a closed discrete subset.
	Any map $\Lambda\to \r^3$ extends to a complete nonflat conformal minimal immersion $X\colon M\to\r^3$ whose Gauss map $M\to\s^2$ omits two (antipodal) values  of the sphere $\s^2$.
\end{corollary}

Theorem \ref{th:GXsimple} and Corollary \ref{cor:gaussmap3} are deduced from an extension result for complete minimal surfaces with prescribed coordinates, see Theorem \ref{th:fixsimple}.
Alarc\'on, Fern\'andez, and L\'opez showed in \cite{AlarconFernandez2011DGA,AlarconFernandezLopez2012CMH,AlarconFernandezLopez2013CVPDE} that one may prescribe all but two of the component functions of a complete conformal minimal immersion $M\to\r^n$. 
On the other hand, Alarc\'on and Castro-Infantes proved in \cite{AlarconCastro2017} the existence of a complete conformal minimal immersion that interpolates any given map at a closed discrete subset of $M$.
In this paper we put together ideas from these two different subjects and obtain the following result.
\begin{theorem}\label{th:fixsimple}
	Let $M$ be an open Riemann surface and $n\ge 3$ be an integer.
	Let $\Lambda\subset M$ be a closed discrete subset and let $\Hgot\colon M\to\r^{n-2}$ be a nonconstant harmonic map.
	For any map $F\colon\Lambda\to\r^2$, there is a complete conformal minimal immersion $X=(X_1,X_2,\ldots,X_n)\colon M\to\r^n$ such that $(X_1,X_2)|_\Lambda=F$ and $(X_3,\ldots,X_n)=\Hgot$.
\end{theorem}
\noindent Here, the assumption on $\Lambda$ is also necessary by the Identity Principle. If one does not care about interpolation, that is, $\Lambda=\emptyset$ in Theorem \ref{th:fixsimple}, then the conclusion follows from \cite[Theorem B]{AlarconFernandezLopez2013CVPDE}.

Theorem \ref{th:fixsimple} is deduced from a more general result, see Theorem \ref{th:fixingCMI}, which ensures not only interpolation but also jet-interpolation of any order at each point $p\in\Lambda$, uniform approximation on a Runge subset of $M$, and prescription of the flux map of the examples, see Section \ref{sec:prelim} for details and definitions. 
Finally, we obtain Theorem \ref{th:GXsimple} as a consequence of Theorem \ref{th:fixingCMI} (see Theorem \ref{th:gasusmap} for a precise statement).


Our proof relies on the method of period dominating sprays of Weierstrass data which has been recently developed in the theories of approximation and interpolation for minimal surfaces in $\r^n$; see Section \ref{sec:lemmas} and specially Lemma \ref{lem:fixharmonic}. These ideas were first used in the study of minimal surfaces in \cite{AlarconForstneric2014IM} and have allowed to construct not only minimal surfaces in $\r^n$ but also complete null holomorphic curves in $\c^n$ with many different global behaviours; see for instance \cite{AlarconLopez2013MA,AlarconForstneric2014IM,AlarconDrinovecForstnericLopez2015Pre,AlarconDrinovecForstnericLopez2015PLMS,AlarconCastro2016} and references therein. 

%
%

\subsection*{Organization of the paper}

Section \ref{sec:prelim} is dedicated to stablish some notation and definitions to the well understanding of the paper. In section \ref{sec:lemmas} we stablish some technical result about the existence of sprays of holomorphic functions, see Lemma \ref{lem:fixharmonic}; it will be very useful in the proof of the main results. Next, we prove Proposition \ref{pro:moving2} in Section \ref{sec:complete} which is crucial to ensure completeness of the examples. Finally, Section \ref{sec:main} contains the proof of Theorem \ref{th:fixingCMI} and Theorem \ref{th:gasusmap} which trivially imply Theorems \ref{th:fixsimple} and \ref{th:GXsimple}.


\section{Preliminaries}\label{sec:prelim}

We denote $\imath=\sqrt{-1}$ and $\z_+=\{0,1,2,\ldots\}$.
Given an integer $n\in\n=\{1,2,3,\ldots\}$ and $\k\in\{\r,\c\}$, we denote by $||\cdot||$ and $\length(\cdot)$ the Euclidean norm and length in $\k^n$, respectively. 
We use the notation $|\cdot|$ for the absolute value (or complex modulus) when $\k=\r$ (or $\k=\c$).

%

Given a smooth connected surface $S$ (possibly with nonempty boundary) and a smooth immersion $X\colon S\to\k^n$, we denote by $\dist_X\colon S\times S\to\r_+$ the Riemannian distance induced on $S$ by the Euclidean metric of $\k^n$ via $X$; i.e.,
\[
	\dist_X(p,q):=\inf\{\length(X(\gamma)): 
	\text{$\gamma\subset S$ arc connecting $p$ and $q\}$},\quad p,q\in S.
\]

An immersed open surface $X\colon S\to \k^n$ $(n\ge 3)$ is said to be {\em complete} if the image by $X$ of any divergent arc on $S$ 
has infinite Euclidean length; equivalently, if the Riemannian metric on $S$ induced by $\dist_X$ is complete in the classical sense.


\subsection{Riemann surfaces and spaces of maps}\label{ss:RS}

Throughout the paper every Riemann surface will be considered connected unless the contrary is indicated.

Let $M$ be an open Riemann surface. Given a subset $A\subseteq M$ we denote by $\Cscr(A,\k^n)$ the space of continuous functions $A\to\k^n$. We also denote by $\Oscr(A)$ 
the space of functions $A\to \c$ which are holomorphic 
on an unspecified open neighbourhood of $A$ in $M$.
We use the notation $\Ascr(A)$ for the space of continuous functions on $A$ which are holomorphic on $\mathring A$, that is, $\Ascr(A)=\Oscr(\mathring A)\cap\Cscr(A,\c)$.

Similarly, by a {\em conformal minimal immersion} $A\to\r^n$ we mean a map $A\to\r^n$ which extends to a conformal minimal immersion in an unspecified neighbourhood of $A$.



Throughout the paper, we deal with some special subsets of an open Riemann surface: Runge subsets. 
A compact subset $K$ of an open Riemann surface $M$ is said to be {\em Runge} (also called {\em holomorphically convex} or {\em $\Oscr(M)$-convex}) if every continuous function $K\to\c$, holomorphic in the interior $\mathring K$, may be approximated uniformly on $K$ by holomorphic functions $M\to\c$. By the Runge-Mergelyan theorem \cite{Runge1885AM,Mergelyan1951DAN,Bishop1958PJM} this is equivalent to that $M\setminus K$ has no relatively compact connected components in $M$. 
%
The following particular kind of Runge subsets will play a crucial role in our argumentation.
\begin{definition}\label{def:admissible}
	A nonempty compact subset $S$ of an open Riemann surface $M$ is called {\em admissible} if it is Runge in $M$ and of the form $S=K\cup \Gamma$, where $K$ is the union of finitely many pairwise disjoint  smoothly bounded compact domains in $M$ and $\Gamma:=\overline{S\setminus K}$ is a finite union of pairwise disjoint smooth Jordan arcs meeting $K$ only in their endpoints (if at all) such that their intersections with the boundary $bK$ of $K$ are transverse. 
\end{definition}

Despite most of the upcoming results, in particular Lemma \ref{lem:fixharmonic}, may be proved for any admissible subset, we deal in this paper with {\em very simple admissible} subsets.
This kind of subsets are enough for the purpose of the paper and make the proofs more readable. 
They were first introduced in \cite{AlarconCastro2017}.
\begin{definition}\label{def:simple}
	Let $M$ be an open Riemann surface.
	An admissible subset $S=K\cup\Gamma\subset M$ (see Definition \ref{def:admissible}) will be said {\em simple} if $K\neq\emptyset$, every component of $\Gamma$ meets $K$, $\Gamma$ does not contain closed Jordan curves, and every closed Jordan curve in $S$ meets only one component of $K$. Further, $S$ will be said {\em very simple} if it is simple, $K$ has at most one non-simply connected component $K_0$, which will be called the {\em kernel component} of $K$, and every component of $\Gamma$ has at least one endpoint in $K_0$; in this case we denote by $S_0$ the component of $S$ containing $K_0$ and call it the {\em kernel component} of $S$.
\end{definition}
Note that a connected admissible subset $S=K\cup\Gamma$ in an open Riemann surface $M$ is very simple if and only if $K$ has $m\in\n$  components $K_0,\ldots, K_{m-1}$, where $K_i$ is simply-connected for every $i>0$, and $\Gamma=\Gamma'\cup\Gamma''\cup(\bigcup_{i=1}^{m-1} \gamma_i)$ where
\begin{itemize}
	\item $\Gamma'$ consists of components of $\Gamma$ with both endpoints in $K_0$,
	\item $\Gamma''$ consists of components of $\Gamma$ with an endpoint in $K_0$ and the other in $M\setminus K$, and
	\item $\gamma_i$ is a component of $\Gamma$ connecting $K_0$ to $K_i$ for each $i=1,\ldots,m-1$.
\end{itemize} 
Observe that, in such a case, $K_0\cup \Gamma'$ is a strong deformation retract of $S$.
In general, a very simple admissible subset $S\subset M$ is of the form $S=(K\cup\Gamma)\cup K'$ where $K\cup\Gamma$ is a connected very simple admissible subset and $K'\subset M\setminus (K\cup\Gamma)$ is a (possibly empty) union of finitely many pairwise disjoint smoothly bounded compact disks. 

\subsection{Weierstrass representation formula}\label{ss:wrep}
%

Let $M$ be an open Riemann surface and let $X\colon M\to\r^n$, $n\ge 3$, be a conformal minimal immersion. Denoting by $\di$ the complex linear part of the exterior differential $d=\di+\overline\di$ on $M$ (here $\overline\di$ denotes the antilinear part), we have that the $1$-form
\(
\partial X=(\di X_1,\ldots,\di X_n),
\)
assuming values in $\c^n$, is holomorphic, has no zeros, and satisfies 
\(
\sum_{j=1}^n (\di X_j)^2=0.
\)
Furthermore, the real part $\Re(\di X)$ of $\di X$ is an exact $1$-form on $M$ and the {\em flux map} (or simply, the {\em flux}) of $X$ is a group homomorphism denoted by $\Flux_X\colon \Hcal_1(M;\z)\to\r^n$, of the first homology group of $M$ with integer coefficients. It is defined by
\begin{equation}\label{eq:flux}
\Flux_X(\gamma)=\Im\int_\gamma\di X=-\imath\int_\gamma \di X,\quad \gamma\in \Hcal_1(M;\z),
\end{equation}
where $\Im$ denotes imaginary part.
On the other hand, every holomorphic $1$-form $\Phi=(\phi_1,\ldots,\phi_n)$ with values in $\c^n$, vanishing nowhere on $M$, satisfying the nullity condition
\begin{equation}\label{eq:nullity}
\sum_{j=1}^n (\phi_j)^2=0\quad \text{everywhere on $M$},
\end{equation}
and whose real part $\Re(\Phi)$ is exact on $M$, determines a conformal minimal immersion $X\colon M\to\r^n$ with $\di X=\Phi$ by the classical Enneper-Weierstrass, or simply Weierstrass, representation formula:
\begin{equation}\label{eq:Weierstrass}
X(p)=x_0+\Re\int_{p_0}^p 2\, \Phi,\quad p\in M,
\end{equation}
for any fixed base point $p_0\in M$ and initial condition $X(p_0)=x_0\in\r^n$. 

If we are given a holomorphic $1$-form $\theta$ never vanishing on $M$ (such exists by Oka-Grauert \cite{Grauert1957MA,Grauert1958MA}, see also Gunning and Narasimhan \cite{GunningNarasimhan1967MA}), then any holomorphic $1$-form $\Phi$ satisfying \eqref{eq:nullity}
may be written as $\Phi=f\theta$ where $f\colon M\to\Agot$ is a holomorphic function and
\begin{equation}\label{eq:nullquadric}
\Agot:=\{(z_1,\ldots,z_n)\in\c^n : z_1^2+\ldots z_n^2=0\}.
\end{equation}
The hypercuadric $\Agot$ is called the null quadric and $\Agot\setminus\{0\}$ is an Oka manifold (see \cite[Example 4.4]{AlarconForstneric2014IM}). This fact allows us to apply all the theory of maps from Stein manifolds into Oka manifolds to the study of minimal surfaces; recall that an open Riemann surface is an Stein manifold. See \cite{Forstneric2017book} for a specialized book on the field and \cite{AlarconForstneric2017Survey} for how these techniques apply to the study of minimal surfaces. 

Gunning and Narasimhan proved in \cite{GunningNarasimhan1967MA} that every open Riemann surface admits a never vanishing exact holomorphic $1$-form. Kusunoki and Sainouchi generalized this result to the existence of a holomorphic $1$-form with given divisor and periods on any open Riemann surface, see \cite{KusunokiSainouchi1971JMKU}.

We shall use the following result which easily follows by putting together the ideas from \cite{KusunokiSainouchi1971JMKU} and \cite{AlarconCastro2017}.
\begin{lemma}\label{lem:specialtheta}
	Let $M$ be an open Riemann surface and $\Lambda\subset M$ be a closed discrete subset. Let $\pgot\colon \Hcal_1(M;\z)\to\c^n$ be a group morphism and $F\colon \Lambda\to\c^n$ be a map. Additionally, for any $p\in\Lambda$, let $\gamma_p\subset M$ be a smooth oriented Jordan arc connecting a fixed point $p_0\in M\setminus \Lambda$ to $p\in\Lambda$.
	Then, there exists a never vanishing holomorphic $1$-form $\theta$ on $M$ such that
	\begin{enumerate}[\rm (i)]
		\item $\displaystyle \int_{\gamma}\theta = \pgot(\gamma)$ for any closed curve $\gamma\subset M$, and
		\item $\displaystyle\int_{\gamma_p} \theta= F(p)$  for any $p\in\Lambda$.
	\end{enumerate}
\end{lemma}
\begin{proof}[Sketch of the proof]
	%
	Fix $p_0\in M\setminus \Lambda$. We consider an exhaustion of $M$ by smoothly bounded connected compact Runge subsets $\{K_j\}_{j\in\n}$ with $p_0\in K_1$ and $p_j\cap bK_j=\emptyset$ for any $j\in\n$. We also assume that for any $j\in\n$ and all $p\in \Lambda \cap K_j$ we have that $\gamma_p\subset \mathring{K_j}$. 
	
	Combining the method from  \cite[proof of Theorem $1$]{KusunokiSainouchi1971JMKU} and \cite{AlarconCastro2017} we construct a sequence of holomorphic functions $\{h_j\colon K_j\to\c_* \}_{j\in\n}$ with the following properties:
	\begin{itemize}
		\item $\displaystyle |h_j(p)-h_{j-1}(p)|< \epsilon_j$ for any $p\in K_{j-1}$, where $\{\epsilon_j\}$ is a sequence of sufficiently small positive numbers (depending on $h_{j-1}$ at each step).
		\smallskip
		\item $\displaystyle \int_{\gamma}h_j\omega = \pgot(\gamma)$ for any closed curve $\gamma\subset K_j$.
		\item $\displaystyle\int_{\gamma_p} h_j\omega= F(p)$ for any $p\in K_j$.
	\end{itemize}
 Therefore, if the sequence $\{\epsilon_j\}_{j\in\n}$ is chosen decreasing to zero fast enough, then Hurwitz theorem ensures that there exists a limit map $h=\lim_{j\to\infty} h_j\colon M\to\c_*$. Thus, $\theta=h\omega$ is the desired $1$-form.
\end{proof}

\subsection{Jets of maps}

Let $\Mcal$ and $\Ncal$ be smooth manifolds without boundary, $x_0\in \Mcal$ be a point, and $f,g\colon \Mcal\to\Ncal$ be smooth maps. $f$ and $g$ have a {\em contact of order $k\in\z_+$} at the point $x_0$ if their Taylor series at this point coincide up to the order $k$. An equivalence class of maps $\Mcal\to\Ncal$ which have a contact of order $k$ at the point $x_0$ is called a {\em $k$-jet}; see e.\ g.\ \cite[\textsection 1]{Michor1980SP} for a basic reference.

In particular, if $\Omega$ is a neighbourhood of a point $p$ in an open Riemann surface $M$ and $f,g\colon \Omega\to\c^n$ are holomorphic functions, then they have a contact of order $k\in\z_+$, or the same $k$-jet, at the point $p$ if and only if $f-g$ has a zero of multiplicity (at least) $k+1$ at $p$. If this is the case, for any distance function ${\sf d}\colon M\times M\to \r_+$ on $M$ (not necessarily conformal) we have 
\begin{equation}\label{eq:O}    
|f-g|(q)=O({\sf d}(q,p)^{k+1})\quad \text{as $q\to p$}.
\end{equation}

Assume that $f,g\colon \Omega\to\r^n$ are harmonic maps, as, for instance, conformal minimal immersions. Then we say that they have a contact of order $k\in\z_+$ (or the same $k$-jet) at the point $p\in\Omega$ if $f(p)=g(p)$ and, if $k>0$, the holomorphic $1$-form $\di(f-g)$ has a zero of multiplicity at least $k$ at $p$. Again, if such a pair of maps $f$ and $g$ have the same $k$-jet at the point $p\in \Omega$ then \eqref{eq:O} formally holds. 

Throughout the paper we shall say that a holomorphic function has a zero of multiplicity $k\in\n$ at a point to mean that the function has a zero of multiplicity {\em at least} $k$ at the point. 
We will follow the same pattern when claiming that two functions have the same $k$-jet or a contact of order $k$ at a point.

\section{Sprays of holomorphic functions}\label{sec:lemmas}

In this section, we  construct sprays of holomorphic maps. 
We combine the arguments in \cite{AlarconFernandezLopez2012CMH,AlarconFernandezLopez2013CVPDE,AlarconCastroLopez2017}.

Let $K$ be a compact admissible subset of an open Riemann surface $M$. Assume that we are given a map $(f_1,f_2)\colon K\to \c^2$ of class $\Ascr(K)$ and a holomorphic function $H\colon M\to \c$ such that
\begin{equation}\label{eq:f1f2H}
	H=f_1^2+f_2^2, \quad \text{on $K$.}
\end{equation}
We consider
\begin{equation}\label{eq:etaH}
\eta\colon K\to\c,\qquad	\eta:=f_1-\imath f_2. 
\end{equation}
By \eqref{eq:f1f2H} and \eqref{eq:etaH}, we have that
\begin{equation}\label{eq:f1f2}
f_1=\frac12 \big(\eta+\frac H\eta\big)  \quad \text{and}\quad f_2=\frac\imath2 \big(\eta-\frac H\eta\big).
\end{equation}
We consider a map $\Phi\colon\Ascr(K)\to\Ascr(K)\times\Ascr(K)$ given by
\begin{equation}\label{eq:Phi}
\Phi(h):=\bigg(\frac12 \big(e^h\eta+\frac {H}{e^h\eta}\big),\frac\imath2 \big(e^h\eta-\frac {H}{e^h\eta}\big)\bigg).
\end{equation}
Notice that $\Phi(0)=(f_1,f_2)$.

The following result is needed.
\begin{lemma}\label{lem:linearindependent}
	Let $M$ be an open Riemann surface and $\theta$ be a holomorphic $1$-form never vanishing on $M$.
	Let $K\subset M$ be a very simple admissible subset, $\Lambda\subset \mathring K$ be a finite subset and $p_0\in \mathring K_0\setminus \Lambda$ be a point, where $K_0$ is the kernel component of $K$ (see Definition \ref{def:simple}).
	Take $k\in\n$ and let $(f_1,f_2)\colon K\to\c^2$ be a map of class $\Ascr(K)$. 
	
	Then, $(f_1,f_2)$ may be approximated uniformly on $K$ by maps $(g_1,g_2)\colon K\to\c^2$ of class $\Ascr(K)$ such that
	\begin{enumerate}[\rm (i)]
		\item $g_1$ and $g_2$ are (complex) linear independent.
		\item $g_1^2+g_2^2=f_1^2+f_2^2$.
		\item $((f_1,f_2)-(g_1,g_2))\theta$ is an exact $1$-form.
		\item $\displaystyle \int_{p_0}^p ((f_1,f_2)-(g_1,g_2))\theta = 0$ for any $p\in\Lambda$.\\
		\item $(g_1,g_2)-(f_1,f_2)$ has a zero of multiplicity $k$ at any $p\in\Lambda$.
	\end{enumerate}
Note that the integral in {\rm (iv)} does not depend on the chosen curve connecting $p_0$ with $p$ when {\rm (iii)} holds.
\end{lemma}
\begin{proof}
	If $f_1$ and $f_2$ are complex linear independent it suffices to choose $(g_1,g_2)=(f_1,f_2)$. If that is not the case, there exists $z_0\in\c_*$ such that $f_2=z_0f_1$.
	
	Define $H\colon K\to\c$ as in \eqref{eq:f1f2H} and $\eta\colon K\to\c$ as in \eqref{eq:etaH}. 
	Notice that equation \eqref{eq:f1f2} holds.
	Set $\Lambda=\{p_1,\ldots,p_m\}$ for $m\in\n$ and take $\gamma_1,\ldots,\gamma_l$ with $l\geq m$ smooth Jordan curves in $K$ such that:
	\begin{itemize}
		\item $\gamma_j$ connects $p_0$ to $p_j$ for every $j=1,\ldots,m$. 
		\item $\gamma_{m+1},\ldots,\gamma_l\subset K$ are closed Jordan curves determining a basis of $\Hcal_1(K,\z)$.
		\item $\gamma_i\cap\gamma_j=\{p_0\}$ for every $i\neq j\in \{1,\ldots,l\}$.
	\end{itemize}
\noindent Recall that $K$ is a very simple admissible subset.
We consider the period-interpolation map $\Pcal=(\Pcal_1,\ldots,\Pcal_l)\colon \Ascr(K)\to(\c^2)^l$ whose $j$-th coordinate, $\Pcal_j\colon\Ascr(K)\to\c^2$, $j=1,\ldots,l$ is given by
\begin{equation}\label{eq:periodmap}
\Pcal_j(h):=\int_{\gamma_j} \big(\Phi(h)-(f_1,f_2)\big)\,\theta,
\end{equation}
where $\Phi\colon \Ascr(K)\to{\Ascr(K)\times\Ascr(K)}$ is the map defined in  \eqref{eq:Phi}. Note that $\Pcal(0)=0$.

Take a holomorphic map $\varphi_0\in\Ascr(K)$ that has a zero of multiplicity at least $k$ at any $p\in\Lambda$. Then, for any nonconstant holomorphic function $\varphi\in\Ascr(K)$ contained in a small open neighbourhood of the zero function in $\Ascr(K)$, the map $e^{\varphi_0\varphi}\eta$ clearly depends holomorphically on $\varphi$. 
If $\varphi=0$, then $\Pcal(\varphi_0\varphi)=\Pcal(0)=0$. Hence, since the space $\Ascr(K)$ is of infinite dimension, there exists a nonconstant function $\varphi\in \Ascr(K)$ arbitrarily close to $0$ such that $\Pcal(\varphi_0\varphi)=0$. For such $\varphi$, the map
\begin{equation}\label{eq:g1g2}
	(g_1,g_2):=\Phi(\varphi_0\varphi)=\bigg(\frac12 \big(e^{\varphi_0\varphi}\eta+\frac {H}{e^{\varphi_0\varphi}\eta}\big),\frac\imath2 \big(e^{\varphi_0\varphi}\eta-\frac {H}{e^{\varphi_0\varphi}\eta}\big)\bigg).	
\end{equation}
solves the lemma.

Indeed, from \eqref{eq:g1g2} and \eqref{eq:Phi} we have that $g_1$ and $g_2$ are linearly independent if and only if $e^\varphi$ and $1$ are, take into account that $f_2=z_0f_1$. Since $\varphi$ is nonconstant, then $g_1$ and $g_2$ are complex linear independent, and so {\rm (i)} holds.
If $\varphi$ is chosen close to $0$ enough, then $(g_1,g_2)$ approximated $(f_1,f_2)$ uniformly on $K$.
A straightforward computation gives {\rm (ii)} since $H=f_1^2+f_2^2$. 
From $\Pcal(e^{\varphi_0\varphi}\eta)=0$ we deduce {\rm (iii)} and {\rm (iv)}, see \eqref{eq:periodmap}.
Finally, taking into account \eqref{eq:g1g2}, \eqref{eq:Phi}, and \eqref{eq:f1f2} we obtain that $(g_1,g_2)-(f_1,f_2)$ has a zero of multiplicity at least the same that $\varphi_0$ at each $p\in\Lambda$; that is, {\rm (v)}.
\end{proof}

We now prove the main technical result of the paper. 

\begin{lemma}\label{lem:fixharmonic}
	Let $M$ be an open Riemann surface and $\theta$ be a holomorphic $1$-form never vanishing on $M$.
	Let $S=K\cup\Gamma\subset M$ be a very simple admissible subset and $L\subset M$ be a smoothly bounded compact domain such that $S\subset\mathring L$ and the kernel component $S_0$ of $S$ is a strong deformation retract of $L$ (see Definition \ref{def:simple}).
	Let $K_0,\ldots, K_{m'}$, $m'\in\z_+$ denote the components of $K$ contained in $S_0$, where $K_0$ is the kernel component of $K$. 
	Let $m\in\z_+$, $m\ge m'$, and let $p_0,\ldots, p_{m}$ be distinct points in $S$ such that $p_i\in \mathring K_i$ for all $i=0,\ldots,m'$ and $p_i\in\mathring K_0$ for all $i=m'+1,\ldots,m$.
	Take $k\in\n$ and let $H\colon L\to\c$ be a nonzero map of class $\Oscr(L)$ and  $f=(f_1,f_2)\colon S\to\c^2$ be a map of class $\Ascr(S)$ such that 
	\begin{equation}\label{eq:f12H}
		f_1^2+f_2^2=H|_S. 
	\end{equation}
	
	Then, $(f_1,f_2)$ may be approximated uniformly on $K$ by maps $\wt f=(\wt f_1,\wt f_2)\colon L\to\c^2$ of class $\Oscr(L)$ such that
	\begin{enumerate}[\rm (i)]
		\item $\wt f_1^2+\wt f_2^2=H $ on $L$.
		\smallskip
		\item $(\wt f-f)\theta$ is an exact $1$-form on $S$, and hence on $L$. Recall that $S$ is a strong deformation retract of $L$.
		\item $\displaystyle \int_{p_0}^{p_j} (\wt f-f)\theta=0$ for any $j=1,\ldots,m$ 
		\item $\wt f-f$ has a zero of multiplicity $k$ at any point $p_1,\ldots,p_m$.
	\end{enumerate}
By {\rm (ii)} the integral in {\rm (iii)} is well defined independently of the chosen curve connecting $p_0$ to $p_j$.
\end{lemma}
\begin{proof}
	Assume that $S$ is connected. If that is not the case, we just add a family of pairwise disjoint Jordan arcs connecting transversely $K_0$ to each connected component of $S$ different from $S_0$, recall Definition \ref{def:simple}. Next, we extend $(f_1,f_2)$ to those arcs continuously and verifying \eqref{eq:f12H} using \cite[Lemma 3.3]{AlarconCastro2017}.
	Assume also by Lemma \ref{lem:linearindependent} that $(f_1,f_2)$ are (complex) linear independent.
	
	We define $\eta$ as in \eqref{eq:etaH} and by \eqref{eq:f12H}, equation \eqref{eq:f1f2} holds, that is, 
	\begin{equation}\label{eq:both}
	\eta=f_1-\imath f_2\quad  \text{ and } \quad (f_1,f_2)=\bigg(\frac12 \big(\eta+\frac H\eta\big),\frac\imath2 \big(\eta-\frac H\eta\big)\bigg).
	\end{equation}
	Take $\gamma_1,\ldots,\gamma_l$ with $l\geq m$ smooth Jordan curves in $S$ such that:
	\begin{enumerate}[\rm (A)]
		\item $\gamma_j$ connects $p_0$ to $p_j$ for every $j=1,\ldots,m$. 
		\item $\gamma_{m+1},\ldots,\gamma_l\subset S_0$ are closed Jordan curves determining a basis of $\Hcal_1(S_0,\z)$, hence of $\Hcal_1(L,\z)$. Recall that $S_0$ is a strong deformation retract of $L$.
		\item $\gamma_i\cap\gamma_j=\{p_0\}$ for every $i\neq j\in \{1,\ldots,l\}$.
	\end{enumerate}
	
	Set $C:=\bigcup_{j=1}^l \gamma_j$ and notice that $C$ is a Runge subset of $M$ which is a strong deformation retract of $L$. Set also $\Lambda=\{p_1,\ldots,p_m\}$.
	
	We consider the period-interpolation map $\Pcal=(\Pcal_1,\ldots,\Pcal_l)\colon \Ascr(S)\to\c^{2l}$ whose $j$-th coordinate, $j=1,\ldots,l$ is given by equation
	\eqref{eq:periodmap} for the curves $\gamma_1,\ldots,\gamma_l$ recently defined and the functions $f_1$ and $f_2$ in the statement of the Lemma, see \eqref{eq:Phi}.
	Notice that for a function $h\in\Ascr(S)$, 
	the value of $\Pcal(h)\in\c^{2l}$ only depends on $h|_C$.

	\begin{claim}\label{cl:hij}
		For each $i=1,2$, there exist functions $h_{i,1},\ldots,h_{i,l}\in\Oscr(L)$ with a zero of order $k$ at each point $p\in\Lambda$ such that the map $\varphi\colon \c^{2l}\to\Oscr(L)$ defined by
	\begin{equation}\label{eq:varphi}
		\varphi(\zeta)= \sum_{i=1}^2\sum_{j=1}^{l}\zeta_{i,j} h_{i,j}(\cdot),
	\end{equation}
	for $\zeta=(\zeta_{i,j})\in\c^{2l}$, $i=1,2$ and $j=1,\ldots,l$, is a dominating spray with respect to $\Pcal$ and satisfies $\Pcal(\varphi(0))=0$. That is to say, $\varphi(0)=0$ and the map $\Pcal\circ \varphi \colon (\c^2)^l \to (\c^2)^l$ is a submersion at $\zeta=0$. In particular, there exists a Euclidean ball $W\subset \c^{2l}$ centred at the origin such that $(\Pcal\circ\varphi)\colon W\to(\Pcal\circ\varphi)(W)$ is a biholomorphism.
		\end{claim}
	\begin{proof}
		Fix $j=1,\ldots,l$. Since $f_1$ and $f_2$ are linear independent, we may choose two points, $p_{i,j}\in\mathring\gamma_j$, $i=1,2$ different from the endpoints, such that the vectors
		\[
			\left\lbrace\bigg( \frac12 \big(\eta+\frac H\eta\big), \frac\imath2 \big(\eta-\frac H\eta\big) \bigg)(p_{i,j})\right\rbrace_{i=1,2}
		\]
		determine a basis of $\c^2$, see \eqref{eq:both}.
		Consider for each $i=1,2$ a continuous function $g_{i,j}\colon C\to\c$ supported on a small neighbourhood of the point $p_{i,j}\in\gamma_j$; the precise value will be specified later. In particular, we have that $g_{i,j}(q)=0$ for each $q\in C\setminus \gamma_j$. 
		
		Given $\zeta=(\zeta_{i,j})\in\c^{2l}$, we consider a continuous map $\wh \varphi(\zeta)\colon C\to\c$, depending holomorphically on $\zeta\in\c^{2l}$, defined by
		\begin{equation}\label{eq:spray}
		\wh\varphi(\zeta)= \sum_{i=1}^2\sum_{j=1}^{l}\zeta_{i,j} g_{i,j}(\cdot).
		\end{equation}
		The differential of $\Pcal\circ\wh\varphi$ with respect to $\zeta_{i,j}$ at $\zeta=0$ may be expressed for any $i=1,2$ as
		\[
		\frac{\partial\Pcal_m\circ\wh\varphi}{\partial \zeta_{i,j}}\bigg|_{\zeta=0}(\zeta)=
		\left\lbrace\begin{matrix}
			(0,0), & j\neq m,\\
			\\
			\displaystyle \int_{\gamma_j}g_{i,j} \left( \frac12 \big(\eta-\frac{H}{\eta} \big),\frac{\imath}{2}\big(\eta+\frac{H}{\eta} \big)\right) \theta, & j=m.
		\end{matrix}\right.
		\]
		
		We claim that we may choose the values of each function $g_{i,j}$ at the curve $\gamma_j$ in order to the vectors
		\begin{equation}\label{eq:vectors}
		\left\lbrace\frac{\partial\Pcal_j\circ\wh\varphi}{\partial \zeta_{1,j}}\bigg|_{\zeta=0}(\zeta),\ \frac{\partial\Pcal_j\circ\wh\varphi}{\partial \zeta_{2,j}}\bigg|_{\zeta=0}(\zeta)\right\rbrace
		\end{equation}
		determine a basis of $\c^2$ and hence the differential of $\Pcal\circ\wh\varphi$ at $\zeta=0$ is surjective. Indeed, let $\gamma_j\colon [0,1]\to\gamma_j$ be a parametrization of the curve $\gamma_j$, we identify $\gamma_j\equiv \gamma_j([0,1])$. For each $i=1,2$ there exists a point $t_{i,j}\in(0,1)$ such that $\gamma_j(t_{i,j})=p_{i,j}$.
		Take a positive number $\rho>0$ small enough such that 
		\[
		[t_{1,j}-\rho,t_{1,j}+\rho]\cap[t_{2,j}-\rho,t_{2,j}+\rho]=\emptyset\quad \text{and}\quad [t_{i,j}-\rho,t_{i,j}+\rho]\subset [0,1],\ i=1,2.
		\]
		We now define each of the function $g_{i,j}$, $i=1,2$ such that they are supported on a small neighbourhood of $t_{i,j}$, that is:
		 \begin{equation}
		 g_{i,j}(t)=0 \quad \text{if}\quad t\in [0,1]\setminus [t_{i,j}-\rho,t_{i,j}+\rho] 
		 \end{equation}
		 and also such that
		 \[
		 	\int_{0}^1g_{i,j}(t)\;dt=\int_{t_{i,j}-\rho}^{t_{i,j}+\rho}g_{i,j}(t)\;dt=1.
		 \]
		 
		 We claim that the functions already defined satisfy the conclusion of the claim except that they are only defined on $C$; we will deal with this later.
		 Indeed, for $\rho>0$ sufficiently small we have that
		 \begin{equation*}\label{eq:derivaties}
		 	\frac{\partial\Pcal_j\circ\wh\varphi}{\partial \zeta_{i,j}}\bigg|_{\zeta=0}(\zeta) = 
		 	\int_{\gamma_j}g_{i,j} \left( \frac12 \big(\eta-\frac{H}{\eta} \big),\frac{\imath}{2}\big(\eta+\frac{H}{\eta} \big)\right) \theta
		 \end{equation*}
		 takes approximately the value
		 \[
		 	\bigg( \frac12 \big(\eta-\frac{H}{\eta} \big),\frac{\imath}{2}\big(\eta+\frac{H}{\eta} \big)\bigg)(t_{i,j}) \ \theta(\gamma_j(t_{i,j}),\dot\gamma_j(t_{i,j}))
		 \]
		 for any $i=1,2$. Since the $1$-form $\theta$ never vanishes on $M$ and since
		 \[
		 \bigg( \frac12 \big(\eta+\frac{H}{\eta} \big),\frac{\imath}{2}\big(\eta-\frac{H}{\eta} \big)\bigg) =
			\bigg( \frac12 \big(\eta-\frac{H}{\eta} \big),\frac{\imath}{2}\big(\eta+\frac{H}{\eta} \big)\bigg)
			\left( \begin{matrix}
			0 & \imath\\
			-\imath & 0
			\end{matrix}\right),
		 \]
		 we have that the vectors 
		 in \eqref{eq:vectors} are a basis of $\c^2$ provided that $\rho>0$ is small enough. 
		 
		 To finish the claim we have to use Mergelyan Theorem with jet-interpolation (see \cite{Forstneric2017book}) to approximate each continuous function $g_{i,j}\colon C\to\c$ by a holomorphic function $h_{i,j}\colon L\to\c$ such that
		 \begin{enumerate}[{\rm {(A)}}]
		 	\item[{\rm {(D)}}] $h_{i,j}$ approximates $g_{i,j}$ uniformly on $C$.
		 	\item[{\rm {(E)}}] $h_{i,j}$ has a zero of order $k \in\n$ at each point $p\in\Lambda$.
		 	\item[{\rm {(F)}}] $h_{i,j}$ vanishes at the points of $L\setminus C$ where $\eta$ vanishes. 
		 \end{enumerate}
	 Recall that the function $g_{i,j}$ vanishes at a neighbourhood of any $p\in\Lambda$. Therefore, if the approximation of {\rm (D)} is close enough, the map $\varphi$ defined in \eqref{eq:varphi}, which is obtained by replacing in \eqref{eq:spray} each of $g_{i,j}$ by $h_{i,j}$, is also dominating with respect to $\Pcal$ and verifies $\varphi(0)=0$.
	\end{proof}

	Mergelyan Theorem with jet-interpolation applied to $\eta\colon S\to\c$ provides a holomorphic function $\wt \eta\colon L\to\c$ such that
	\begin{enumerate}[\rm (A)]
		\item[{\rm {(G)}}] $\wt\eta$ approximates $\eta$ uniformly on $S$.
		\item[\rm (H)] $\wt \eta$ has exactly the same zeros of $\eta$ with the same multiplicity. Hence $H/\wt\eta$ is holomorphic.
		\item[\rm (I)] $\wt \eta-\eta$ has a zero of multiplicity $k$ at any point $p_1,\ldots,p_m$.
	\end{enumerate}

	For a close to zero $\zeta\in\c^{2l}$, we define the map $\Phi_\zeta\colon L\to\c^2$ given by
	\begin{equation}\label{eq:Phizeta}
		\Phi_\zeta:=
		 \left( \frac12 \big(e^{\varphi(\zeta)}\wt\eta+\frac {H}{e^{\varphi(\zeta)}\wt\eta}\big), \frac\imath2 \big(e^{\varphi(\zeta)}\wt\eta-\frac {H}{e^{\varphi(\zeta)}\wt\eta}\big) \right),
	\end{equation}
	recall \eqref{eq:varphi}. Notice that $\Phi_\zeta$ is holomorphic and depends holomorphically on $\zeta\in\c^{2l}$, see property {\rm (H)}.
	Furthermore, provided that the approximations involved in {\rm (D)} and {\rm (G)} are close enough, the map $\wt\Pcal=(\wt\Pcal_1,\ldots,\wt\Pcal_l)\colon \c^{2l}\to \c^{2l}$ given by 
	\begin{equation}\label{eq:wtPcal}
	\wt\Pcal_j(\zeta):=\int_{\gamma_j} \big(\Phi_{\zeta}-(f_1,f_2)\big)\theta
	\end{equation}
	is also a submersion at $\zeta=0$. Therefore, there exists a close to zero $\wt\zeta\in\c^{2l}$ such that $\wt \Pcal(\wt\zeta)=0\in\c^{2l}$.
	Thus, the map $(\wt f_1,\wt f_2)\colon L\to\c^2$ defined by
	\(
		(\wt f_1,\wt f_2):=\Phi_{\wt\zeta}
	\)
	is holomorphic and verifies conditions {\rm (i)}--{\rm (iv)}. Indeed, a straightforward computation gives that $\wt f_1$ and $\wt f_2$ verify condition {\rm (i)}.
	We have that $(\wt f_1,\wt f_2)$ approximates $(f_1,f_2)$ uniformly on $S$ by {\rm (G)} and provided that $\wt \zeta$ is chosen small enough, see \eqref{eq:both} and \eqref{eq:Phizeta}.
	$\wt \Pcal(\wt\zeta)=0$  implies conditions {\rm (ii)} and {\rm (iii)}, see {\rm (A)} and {\rm (B)}. 
	Finally, by  \eqref{eq:varphi} we have that the functions
	\[
	e^{\varphi(\wt\zeta)}\wt\eta - \eta\quad \text{and} \quad \frac{H}{e^{\varphi(\wt\zeta)}\wt\eta}-\frac H\eta
	\]
	have a zero of multiplicity $k$ at each point $p\in\Lambda$, see {\rm (I)} and take into account that $\eta-e^{\varphi(\wt\zeta)}\wt\eta$ vanishes at the points where $\eta$ and $\wt\eta$ do  by {\rm (F)} and {\rm (H)}. Therefore, \eqref{eq:both} and \eqref{eq:Phizeta} imply that $\wt f-f$ has a zero of multiplicity $k$ at any point $p\in\Lambda$, that is, condition {\rm (iv)}.
\end{proof}

\section{Completeness}\label{sec:complete}
This section is dedicated to prove the technical results needed to ensure completeness of the solutions. We start with the following lemma.

\begin{lemma}\label{lem:completeness}
	Let $M$ be an open Riemann surface and $\theta$ be a holomorphic $1$-form never vanishing on $M$.
	Let $K\subset L\subset M$ be smoothly bounded connected compact domains such that $K\subset\mathring L$ and  $L\setminus \mathring K$ is a family of pairwise disjoint compact annuli.
	Let $\Lambda\subset \mathring K$ be a finite subset and take $k\in\n$ and $p_0\in \mathring K\setminus \Lambda$.
	Let also $H\colon L\to\c$ be a nonzero holomorphic function and $f=(f_1,f_2)\colon K\to\c^2$ be a map of class $\Ascr(K)$ such that $f_1^2+f_2^2=H|_K$.
	
	Given $\tau>0$, $(f_1,f_2)$ may be  approximated uniformly on $K$ by maps $\wt f=(\wt f_1,\wt f_2)\colon L\to\c^2$ of class $\Oscr(L)$ such that
	\begin{enumerate}[\rm (i)]
		\item $\wt f_1^2+\wt f_2^2=H$ on $L$.
		\item $(\wt f-f)\,\theta$ is an exact $1$-form on $K$, hence on $L$. Recall that $L\setminus \mathring K$ are annuli.
		\item $\displaystyle \int_{p_0}^{p} (\wt f-f)\,\theta=0$ for any $p\in\Lambda$. 
		\item $\wt f-f$ has a zero of multiplicity $k$ at any point $p\in\Lambda$.
		\item If $\alpha\subset L$ is a curve connecting $p_0$ with $bL$, then \[\displaystyle\int_\alpha \big(|\wt f_1|^2+|\wt f_2|^2+|H|\big)\|\theta\|>\tau.\]
	\end{enumerate}
The integral in {\rm (iii)} is well defined independently of the chosen curve by {\rm (ii)}.
\end{lemma}
\begin{proof}
 For simplicity of exposition we assume that $L\setminus \mathring K$ is connected and hence a single annulus. In the general case, we just apply the upcoming reasoning to each of the annuli.

 Recall that $H\colon L\to\c$ is a nonzero function and $\theta$ never vanishes on $M$. Since  $L\setminus \mathring K$ is an annulus, there exists a family of pairwise disjoint, smoothly bounded, compact disks $D_1,\ldots,D_m\subset L\setminus \mathring K$. For any $i=1,\ldots,m$ we consider a compact disk ${D_i'}\subset \mathring D_i$ such that if $\alpha\subset L$ is an arc connecting $p_0$ with $bL$  and $\alpha$ does not intersect any of the disk $D_i'$ for any $i=1,\ldots,m$, then
%
 \begin{equation}\label{eq:pathnotD}
 \int_{\alpha} |H|\,\|\theta\|>\tau.
 \end{equation}
To construct such disks $D_1,\ldots,D_m
$ we use pieces of a labyrinth of Jorge-Xavier type contained in $\mathring L\setminus K$, see \cite{JorgeXavier1980AM}. For a detailed description of the disks see \cite[Lemma 4.1]{AlarconFernandezLopez2013CVPDE}. 

%

 Set $D:=\bigcup_{i=1}^m D_i$ and $S:=K\cup D$.
 Note that $S$ is a smoothly bounded compact domain. 
 Pick a holomorphic map $g=(g_1,g_2)\colon S\to\c^2$ such that
 \begin{enumerate}[\rm (A)]
 	\item $g_1^2+g_2^2=H$ on $S$.
 	\item $(g_1,g_2)=(f_{1},f_{2})$ on $K$.
 	\item If $\alpha\subset L$ is an arc that intersects any of the disks $D_i'$, then
 	\[\displaystyle \int_{\alpha\cap D}\big(|g_1|^2+|g_2|^2\big)\|\theta\|>\tau.\]
 \end{enumerate}
 Note that $D_i\setminus \mathring D_i'$ is an annulus for any $i=1,\ldots,m$.
 The existence of such a map $(g_1,g_2)$ is clear. Necessarily $(g_1,g_2)=(f_1,f_2)$ on $K$ and take for instance $g_1$ big enough on each $D_i\setminus \mathring D_i'$ in order to property {\rm (C)} holds, recall that $D_i$ is a disk and hence simply connected.
 
 Lemma \ref{lem:fixharmonic} applied to
 \[
 M,\quad \theta,\quad S\subset L, \quad \Lambda\subset K\subset S, \quad p_0\in K,\quad k,\quad H,\quad \text{and} \quad g=(g_1,g_2)
 \]
 ensures that $g=(g_1,g_2)$ may be  approximated uniformly on $S$ by holomorphic maps $\wt f=(\wt f_1,\wt f_2)\colon L\to\c^2$ such that
 	\begin{enumerate}[\rm (D)]
 	\item[\rm (D)] $\wt f_1^2+\wt f_2^2=H $ on $L$.
 	\smallskip
 	\item[\rm (E)] $(\wt f-g)\,\theta$ is an exact $1$-form on $S=K\cup D$, and hence on $L$. Recall that $K$ is a strong deformation retract of $L$ and $D$ a union of pairwise disjoint disks.
 	\smallskip
 	\item[\rm (F)] $\displaystyle \int_{p_0}^{p} (\wt f-g)\,\theta=0$ for any $p\in \Lambda\subset K$. 
 	\item[\rm (G)] $\wt f-g$ has a zero of multiplicity $k$ at any point $p\in\Lambda\subset K$.
 \end{enumerate}

We claim that $\wt f=(\wt f_1,\wt f_2)$ solves the lemma.
Indeed, {\rm (i)} equals {\rm (D)}.
By {\rm (B)} and {\rm (E)} we have {\rm (ii)}.
Furthermore, {\rm (iii)} and {\rm (iv)} follows from {\rm (B)}, {\rm (F)}, and {\rm (G)}.

Finally, let us check that {\rm (v)} holds. Take an arc $\alpha\subset L$ connecting $p_0$ with $bL$, we distinguish cases depending on $\alpha$ crosses some $D_i'$ or not.
If $\alpha$ does not intersect $D_i'$ for any $i=1,\ldots,m$, then by \eqref{eq:pathnotD} we have
\[
	\int_\alpha \big(|\wt f_1|^2+|\wt f_2|^2+|H|\big)\|\theta\|\ge \int_{\alpha} |H|\,\|\theta\|>\tau,
\]
whereas if $\alpha$ intersects any of the disks $D_i'$, $i=1,\ldots,m$, then if the approximation involved is close good enough and by {\rm (C)} and {\rm (B)}, we have that
\[
\int_\alpha \big(|\wt f_1|^2+|\wt f_2|^2+|H|\big)\|\theta\|
\ge \int_{\alpha\cap D} \big(|\wt f_1|^2+|\wt f_2|^2\big)\,\|\theta\|>\tau.
\]
Hence, condition {\rm (v)} holds.
\end{proof}

To finish the section, let us prove the following result. 

\begin{proposition}\label{pro:moving2}
	Let $M$ be an open Riemann surface, $K\subset M$ be a smoothly bounded compact Runge domain, and $\Lambda\subset M$ be a closed discrete subset.
	Let $\theta$ be a holomorphic $1$-form never vanishing on $M$.
	Let $\Omega_p\ni p$, $p\in\Lambda$, be  pairwise disjoint compact neighbourhoods of the points in $\Lambda$. Set $\Omega:=\bigcup_{p\in\Lambda}\Omega_p\subset M$.
	Let $H\colon M\to\c$ be a nonzero holomorphic function, $F\colon\Omega\to\r^2$ be a harmonic map, and assume that there is a holomorphic map $f=(f_1,f_2)\colon K\cup\Omega\to\c^2$ such that $f_1^2+f_2^2=H$ on $K\cup \Omega$ and $f\theta=\partial F$ on $\Omega$.
	Let also $p_0\in \mathring K\setminus \Omega$ be a point, $k\colon \Lambda\to\n$ be a map and  $\pgot\colon \Hcal_1(M;\z)\to\r^2$ be a group morphism such that
	\begin{equation}\label{eq:pgot}
	\pgot(\gamma)=\Im\int_{\gamma} (f_1,f_2)\,\theta \text{ for any closed curve } \gamma\subset K.
	\end{equation}
	
	Given $\epsilon >0$, there exists a holomorphic map $\wt f=(\wt f_1,\wt f_2)\colon M\to\c^2$ such that
	\begin{enumerate}[\rm (i)]
	\item $\wt f_1^2+\wt f_2^2=H $ on $M$.
	\smallskip
	\item 
	$\|(\wt f_1,\wt f_2)(p)-(f_1,f_2)(p)\|<\epsilon$ for any $p\in K$.
	\item $\Re (\wt f\theta)$ is an exact $1$-form and $\displaystyle\Im\int_{\gamma} \wt f\theta=\pgot(\gamma)$ for any closed curve $\gamma$ on $M$.
	\item $\displaystyle \Re\int_{p_0}^{p} \wt f\theta=F(p)$ for any $p\in\Lambda$.
	\item $\wt f-f$ has a zero of multiplicity $k(p)$ at any point $p\in\Lambda$.
	\item $\big(|\wt f_1|^2+|\wt f_2|^2+|H|\big)\|\theta\|^2$ is a complete Riemannian metric on $M$ except for the points where $f_1$ and $f_2$ (and hence $H$) vanish simultaneously. That is, $|\wt f_1|^2+|\wt f_2|^2+|H|>0$ at any point $p\in M\setminus \{q\in K\cup \Omega : f_1(q)=f_2(q)=0\}$ and
	if $\alpha$ is a divergent arc in $M$, then $\displaystyle \int_\alpha \big(|\wt f_1|^2+|\wt f_2|^2+|H|\big)\|\theta\|$ is infinite.
	\end{enumerate}
\end{proposition}
\begin{proof}
	Up to slightly enlarging $K$ and shrinking the subsets $\Omega_p$ if necessary, we may assume without lost of generality that $bK\cap\Lambda=\emptyset$ and that $\Omega_p\subset \mathring{K}$ or $\Omega_p\cap K=\emptyset$ for any $p\in\Lambda$.

Set $K_0:=K$ and let $\{K_j\}_{j\in\n}$ be an exhaustion of smoothly bounded compact Runge domains on $M$ such that
\[
K_0\Subset K_1\Subset\cdots\Subset \bigcup_{j\in\n}K_j=M \quad\text{ and }\quad \chi(K_j\setminus \mathring K_{j-1})\in\{-1,0\},
\]
where $\chi(\cdot)$ denotes the Euler characteristic. 
We assume that $bK_j\cap \Lambda=\emptyset$ and that for any $p\in\Lambda$, we have that $\Omega_p\subset \mathring{K_j}$ or $\Omega_p\cap K_j=\emptyset$ for $j\in\n$. 
We also assume that $H$ never vanishes at $bK_j$ for any $j\in\n$.
The existence of such a sequence is guaranteed by basic topological arguments.

Set $V=\{p\in M\setminus (K\cup \Omega): H(p)=0 \}$. For any $p\in V$  we consider a simply connected neighbourhood $A_p$ of $p$ in $M$ such that $A:=\cup_{p\in V} A_p$ is a union of pairwise disjoint disks. Recall that $H$ is a nonzero holomorphic function. 
Up to shrinking the $A_p$'s if necessary we may assume that $A\cap bK_j=\emptyset$ for any $j\in\n$.
We extend $f=(f_1,f_2)$ to $A$ such that $f_1^2+f_2^2=H$ and $f$ does not vanish at $A$.

Set $(f_{0,1},f_{0,2}):=(f_1,f_2)\colon K_0\cup \Omega \cup A\to\c^2$. Given a sequence of positive numbers $\{\epsilon_j\}_{j\in\n}$ which will be specified later, we recursively construct a sequence of holomorphic maps $\{ (f_{j,1},f_{j,2})\colon K_j\cup\Omega\cup A\to\c^2 \}_{j\in\n}$ with the following properties for any $j\in\n$:
	\begin{enumerate}[\rm (i$_j$)]
		\item $f_{j,1}^2+f_{j,2}^2=H$ on $K_j\cup A$.
		\medskip
		\item $\| (f_{j,1},f_{j,2})(p)-(f_{j-1,1},f_{j-1,2})(p) \|<\epsilon_j$ for any $p\in K_{j-1}\cup A$.
		\smallskip
		\item $\displaystyle\int_{\gamma} (f_{j,1},f_{j,2})\,\theta=\imath \pgot(\gamma)$ for any closed curve $\gamma\subset K_j$.
		\item $\displaystyle\Re\int_{p_0}^p  (f_{j,1},f_{j,2})\,\theta =F(p)$ for any $p\in \Lambda\cap K_j$. By {\rm (iii$_j$)} this is independent of the chosen path connecting $p_0$ with $p$.
		\item $(f_{j,1},f_{j,2})-(f_{j-1,1},f_{j-1,2})$ has a zero of multiplicity $k(p)$ at any $p\in \Lambda$.
		\medskip
		\item If $\alpha\subset K_j$ is an arc connecting $p_0$ with $bK_j$ then
		\[
			\int_\alpha \big(|f_{j,1}|^2+|f_{j,2}|^2+|H|\big)\|\theta\|>j.
		\] 
	\end{enumerate}
	
	Assume that we have such a sequence. We choose the positive numbers $\{\epsilon_j \}_{j\in\n}$ decreasing to zero fast enough such that properties {\rm (ii$_j$)} for $j\in\n$ ensure that there exists a limit holomorphic map
	\[
		\wt f\colon M\to\c^2,\qquad \wt f=(\wt f_1,\wt f_2):=\lim\limits_{j\to+\infty}(f_{j,1},f_{j,2})
	\]
	and that $\|(\wt f_1,\wt f_2)(p)-(f_1,f_2)(p)\|\leq \sum_{j=1}^\infty \epsilon_j<\epsilon$ for any $p\in K$.
	Thus, condition {\rm (ii)} holds.
	Clearly, properties {\rm (i$_j$)} imply {\rm (i)}.
	{\rm (iii)} is a consequence of {\rm (iii$_j$)}; indeed, the real part of the equation implies that $\Re (\wt f \theta)$ is an exact $1$-form whereas the imaginary part ensures $\Im\int_{\gamma} \wt f\theta=\pgot(\gamma)$.
	Additionally, {\rm (iv)} is a consequence of {\rm (iv$_j$)}, $j\in\n$, whilst {\rm (v)} follows from {\rm (v$_j$)}, $j\in\n$. 
	Finally, let us check that {\rm (vi)} holds. If the approximations involved in {\rm (ii$_j$)} are good enough then $\wt f=(\wt f_1,\wt f_2)$ never vanishes on $A$ since $f$ does not and by Hurwitz's theorem $(\wt f_1,\wt f_2)$ vanishes at a point in $K\cup \Omega$ only if $(f_1,f_2)$ does. Hence $\big(|\wt f_1|^2+|\wt f_2|^2+|H|\big)\|\theta\|^2$ is a Riemannian metric on $M$ except for the points in $K\cup \Omega$ where $f=(f_1,f_2)$ vanishes, recall that $\theta$ is a holomorphic $1$-form never vanishing on $M$. Further, properties {\rm (vi$_j$)} for any $j\in\n$ ensure that any divergent path has infinite length.
		
	Let us construct such a sequence to finish the proof.
	First, the base of the induction is given by $(f_{0,1},f_{0,2})$.
	Clearly, {\rm (i$_0$)} holds.
	Conditions {\rm (ii$_0$)}, {\rm (v$_0$)}, and {\rm (vi$_0$)} are vacuous.
	Condition {\rm (iii$_0$)} holds since \eqref{eq:pgot} and {\rm (iv$_0$)} follows since $f\theta=\partial F$ on $K=K_0$.
	
	Assume now that we have already constructed the term $(f_{j-1,1},f_{j-1,2})\colon K_{j-1}\cup \Omega\cup A\to\c^2$ for some $j\in\n$ and let us construct $(f_{j,1},f_{j,2})\colon K_j\cup\Omega\cup A\to\c^2$.
	Notice that, since $\Lambda$ is closed and discrete, we have that $\Lambda_j:=\Lambda\cap K_j$ is empty or finite. 
	First of all, it is clear that $(f_{j,1},f_{j,2})|_{(\Omega\cup A)\setminus K_j}:=(f_{j-1,1},f_{j-1,2})$ verifies {\rm (ii$_j$)} at $p\in A\setminus K_j$ and {\rm (v$_j$)} for any $p\in\Lambda\setminus \Lambda_j$.
	Hence, it is enough to construct $(f_{j,1},f_{j,2})$ on $K_j$.	
	
	Let $S\subset K_j$ be a connected very simple admissible Runge subset obtained by attaching finitely many pairwise disjoint Jordan arcs to $K_{j-1}\cup ((\Omega\cup A)\cap K_j)$ so that $S$ is a strong deformation retract of $K_j$ and $\Lambda_j \subset \mathring S$.
	Recall that $A\cap (K_j\setminus \mathring K_{j-1})$ is a union of pairwise disjoint disks.
	Observe that some of the attached arcs describe the topology of $K_j\setminus K_{j-1}$, if nontrivial, whilst the others connect $K_{j-1}$ to any $\Omega_p$ for each $p\in \Lambda\cap (K_j\setminus K_{j-1})$ and to any $A_p$ for any $p\in A\cap (K_j\setminus \mathring K_{j-1})$.
	
	Next, we extend $(f_{j-1,1},f_{j-1,2})$ to $S$ continuously using \cite[Lemma 3.3]{AlarconCastro2017} such that the following properties hold:
	\begin{enumerate}[\rm (a)]
		\item $f_{j-1,1}^2+f_{j-1,2}^2=H$ on $S$, recall that $(f_{j-1,1}^2+f_{j-1,2}^2)(p)=H(p)$ for any $p\in\Omega\cup A$,
		\item if $\gamma\subset S$ is a closed arc, then
		\[
		\int_{\gamma} (f_{j-1,1},f_{j-1,2})\,\theta =\imath \pgot(\gamma),
		\]
		\item and if $\gamma_p\subset S$ is an arc connecting $p_0$ with $p\in \Lambda_j$, then
		\[
		\displaystyle \Re\int_{\gamma_p} (f_{j-1,1},f_{j-1,2})\,\theta =F(p).
		\]
	\end{enumerate}

	Take a smoothly bounded compact neighbourhood $L$ of $S$ which is a strong deformation retract of $K_j$. Recall that $S$ is a strong deformation retract of $K_j$.
	Lemma \ref{lem:fixharmonic} provides a holomorphic map $(g_{1},g_{2})\colon L\to\c^2$ such that
		\begin{enumerate}[\rm (A1)]
		\item[\rm (A1)] $g_1^2+g_2^2=H $ on $L$.
		\item[\rm (A2)] $\|(g_1,g_2)(p)-(f_{j-1,1},f_{j-1,2})(p)\|<\epsilon_j/2$ for any $p\in S$.
		\item[\rm (A3)] $\big((g_1,g_2)-(f_{j-1,1},f_{j-1,2})\big)\theta$ is an exact $1$-form on $S$, hence on $L$, see \rm (b). 
		\item[\rm (A4)] $\displaystyle \int_{p_0}^{p} \big((g_1,g_2)-(f_{j-1,1},f_{j-1,2})\big)\theta=0$ for any $p\in\Lambda_j$. 
		\item[\rm (A5)] $(g_1,g_2)-(f_{j-1,1},f_{j-1,2})$ has a zero of multiplicity (at least) $k(p)$ at any point $p\in\Lambda_j$.
	\end{enumerate}
	Therefore, Lemma \ref{lem:completeness} applied to
	\[
	M,\quad \theta,\quad L\subset K_j,\quad \Lambda_j\subset L,\quad k, \quad H, \quad (g_1,g_2), \quad \text{ and }\quad j>0
	\]
	provides a holomorphic map $(f_{j,1},f_{j,2})\colon K_j\to\c^2$ such that 
	\begin{enumerate}[\rm (B1)]
		\item $f_{j,1}^2+f_{j,2}^2=H$ on $K_j$.
		\item $\|(f_{j,1},f_{j,2})(p)-(g_1,g_2)(p)\|<\epsilon_j/2$ for any $p\in L$.
		\item $((f_{j,1},f_{j,2})-(g_1,g_2))\theta$ is an exact $1$-form on $L$, hence on $K_j$. 
		\item $\displaystyle \int_{p_0}^{p} ((f_{j,1},f_{j,2})-(g_1,g_2))\theta=0$ for any $p\in\Lambda_j$. 
		\item $(f_{j,1},f_{j,2})-(g_1,g_2)$ has a zero of multiplicity $k(p)$ at any point $p\in\Lambda_j$.
		\item If $\alpha\subset K_j$ is a curve connecting $p_0$ with $bK_j$, then
		\[
		\displaystyle\int_\alpha \big(| f_{j,1}|^2+| f_{j,2}|^2+|H|\big)\|\theta\|>j.
		\]
	\end{enumerate}
	
	We claim that the pair $(f_{j,1},f_{j,2})$ satisfies the conclusion. Indeed, it is clear that {\rm (i$_j$)} equals {\rm (B1)} and {\rm (vi$_j$)} equals {\rm (B6)}.
	Additionally, {\rm (A2)} and {\rm (B2)} imply {\rm (ii$_j$)}.
	{\rm (v$_j$)} follows from {\rm (A5)} and {\rm (B5)}.
	On the other hand, by {\rm (iii$_{j-1}$)} we have that the real part of $(f_{j-1,1},f_{j-1,2})\theta$ is an exact $1$-form, hence {\rm (A3)} and {\rm (B3)} implies {\rm (iii$_j$)}, recall that $L$ is a strong deformation retract of $K_j$.
	Finally, {\rm (iv$_j$)} follows from {\rm (c)}, {\rm (A4)}, and {\rm (B4)}.
\end{proof}

\begin{remark}\label{rem:nojet}
	Proposition \ref{pro:moving2} holds simpler if one just ensure interpolation but not jet interpolation. That is, if the map $k(p)=0$ for any $p\in\Lambda$, then it is not necessary to consider the subsets $\Omega_p$'s defining the jets, and so that $F\colon \Lambda\to\r^3$ is just a map.
	In such a case, Proposition \ref{pro:moving2} ensures conditions {\rm (i)}--{\rm (iv)} and {\rm (vi)}; whilst {\rm (v)} does not make sense.
\end{remark}

\section{Main results and applications}\label{sec:main}

In this section we  prove Theorem \ref{th:fixsimple}, Theorem \ref{th:GXsimple} and Corollary \ref{cor:gaussmap3}.
Let $\pi\colon\r^n\to\r^{n-2}$ be the projection into the $n-2$ last coordinates, that is, if $(x_1\ldots,x_n)\in\r^n$ then $\pi(x_1\ldots,x_n)=(x_3\ldots,x_n)\in\r^{n-2}$.

As stated, Theorem \ref{th:fixsimple} is trivially a consequence of the following result.
\begin{theorem}\label{th:fixingCMI}
	Let $M$ be an open Riemann surface, $K\subset M$ be a smoothly bounded compact Runge domain, and $\Lambda\subset M$ be a closed discrete subset. 
	Let $\Omega_p\ni p$, $p\in\Lambda$, be pairwise disjoint compact neighbourhoods of the points in $\Lambda$. 
	Set $\Omega=\bigcup_{p\in\Lambda}\Omega_p\subset M$.
	Let $\Hgot\colon M\to\r^{n-2}$ be a nonconstant harmonic map and $X\colon K\cup\Omega\to\r^n$ be a conformal minimal immersion 
	 such that $\pi\circ X=\Hgot$ on $K\cup \Omega$.
	Let $k\colon \Lambda\to\n$ be a map and $\pgot\colon \Hcal_1(M;\z)\to\r^n$ be a group morphism such that $\Flux_X=\pgot$ on $K$ and 
	\begin{equation}\label{eq:fluxcondition}
		\Im\int_{\gamma} \partial\Hgot = \pi \circ \pgot(\gamma), \text{ for any closed curve $\gamma\subset M$.}
	\end{equation}
	
	Then, $X$ may be  approximated uniformly on $K$ by complete conformal minimal immersions $\wt X=(\wt X_1,\ldots,\wt X_n)\colon M\to\r^n$ such that
	\begin{enumerate}[\rm (a)]
		\item $(\wt X_3,\ldots,\wt X_n)=\Hgot$.
		\item $\wt X$ and $X$ have a contact of order $k(p)\in\n$ at any point $p\in\Lambda$.
		\item $\Flux_{\wt X}=\pgot$ on $M$.
	\end{enumerate}
\end{theorem}
\begin{proof}
	Let $\theta$ be a holomorphic $1$-form never vanishing on $M$. Set $\partial X=f\theta$ where $f=(f_1,\ldots,f_n)\colon K\cup \Omega\to \Agot_*=\Agot\setminus\{0\}$ is a holomorphic map, see \eqref{eq:nullquadric}. Set also $\partial\Hgot=h\theta$ where $h=(h_3,\ldots,h_n)\colon M\to\c^{n-2}$ is a holomorphic map. Notice that $(f_3,\ldots,f_n)=(h_3,\ldots,h_n)$ on $K\cup \Omega$.
	
	Write $X=(X_1,\ldots,X_n)$ and $\pgot=(\pgot_1,\ldots,\pgot_n)$. We consider the holomorphic map $H\colon M\to\c$ defined by
	\begin{equation}\label{eq:H}
	H:=-(h_3^2+\ldots+h_n^2), 
	\end{equation}
	which is nonzero since $\Hgot$ is nonconstant. Fix also a point $p_0\in \mathring  K\setminus \Omega$. Proposition \ref{pro:moving2} applied to the data 
	\[
	K\subset M,\ \Lambda\subset\Omega \subset M,\ \theta, \ p_0\in K, \ 
	 H, \ F=(X_1,X_2), \ f=(f_1,f_2),\ k,  \ \text{ and }\ (\pgot_1,\pgot_2)
	\]
	provides, given $\epsilon >0$, a holomorphic map $\wt f=(\wt f_1,\wt f_2)\colon M\to\c^2$ with the following properties:
	\begin{enumerate}[\rm (i)]
		\item $\wt f_1^2+\wt f_2^2=H $ on $M$.
		\item $\|(\wt f_1,\wt f_2)(p)-(f_1,f_2)(p)\|<\epsilon$ for any $p\in K$. 
		\item $\Re (\wt f\theta)$ is an exact $1$-form and $\displaystyle\Im\int_{\gamma}  (\wt f_1,\wt f_2)\,\theta=(\pgot_1,\pgot_2)(\gamma)$ for any closed curve $\gamma$ on $M$.
		\item $\displaystyle\Re \int_{p_0}^{p} (\wt f_1,\wt f_2)\,\theta=(X_1,X_2)(p)$ for any $p\in\Lambda$.
		\item $(\wt f_1,\wt f_2)-( f_1, f_2)$ has a zero of multiplicity $k(p)$ at any point $p\in\Lambda$.
		\item $\big(|\wt f_1|^2+|\wt f_2|^2+|H|\big)\|\theta\|^2$ is a complete Riemannian metric on $M$.
	\end{enumerate}
	Notice that Proposition \ref{pro:moving2} ensures that the metric involved in {\rm (vi)} has no singular points since $X$ is a conformal minimal immersion on $K$ and so that, $f_1$, $f_2$, and $H$ have no common zeros.

	We claim that the map $\wt X=(\wt X_1,\ldots,\wt X_n)\colon M\to\r^n$ defined by 
	\[
		(\wt X_1,\wt X_2)=X(p_0)+\Re\int_{p_0}^p  (\wt f_1,\wt f_2)\,\theta ,\quad \text{and} \quad (\wt X_3,\ldots,\wt X_n)=\Hgot
	\]
	is a conformal minimal immersion and solves the Theorem. Indeed, clearly {\rm (a)} is satisfied.
	By first part of property {\rm (iii)}, $\wt X$ is a well defined map.
	By {\rm (i)}, {\rm (vi)}, and \eqref{eq:H}, we have that $(\wt f_1,\ldots,\wt f_n)$ is a holomorphic map assuming values into $\Agot_*$, hence $\wt X$ is a conformal minimal immersion.
	Condition {\rm (ii)} ensures that $X$ is approximated uniformly on $K$ by $\wt X$.
	{\rm (b)} is implied by {\rm (iv)} and {\rm (v)}.
	Second part of {\rm (iii)} implies {\rm (c)}, see \eqref{eq:fluxcondition}.
	Finally,
	{\rm (vi)} ensures that $\wt X$ is complete.
\end{proof}

\begin{proof}[Proof of Corollary \ref{cor:gaussmap3}]
	Let $M$ be an open Riemann surface and $\Lambda\subset M$ be a closed discrete subset.
	Let also $F=(F_1,F_2,F_3)\colon \Lambda\to\r^3$ be a map. By Forstneri\v c \cite[Theorem 2.1]{Forstneric2003Acta},  there exists a harmonic function $h\colon M\to\r$ such that
\begin{itemize} 
	\item $h$ has no critical points, that is, $dh\neq 0$ on $M$, and
	\item $h|_\Lambda=F_3$.
\end{itemize}
Theorem \ref{th:fixsimple} applied to the map $F=(F_1,F_2)$ and $\Hgot=h$ provides a complete conformal minimal immersion $X=(X_1,X_2,X_3)\colon M\to\r^3$ such that $X_3=h$.
Note that Theorem \ref{th:fixsimple} gives us that $X$ is nonflat, for instance add to $\Lambda$ four points with images by $F$ contained in no plane.
Therefore, its Gauss map $M\to\s^2$ misses the north and south poles of the sphere. Equivalently, its generalized Gauss map $G_X\colon M\to \cp^2$, see \eqref{eq:ggm}, fails to intersect the complex plane $\{(z_1,z_2,z_3)\in\cp^2 : z_3=0\}$. 
\end{proof}

Theorem \ref{th:GXsimple} follows from the following result.

\begin{theorem}\label{th:gasusmap}
	Let $M$ be an open Riemann surface and $\Lambda\subset M$ be a closed discrete subset.
	Let $Y\colon \Lambda\to \r^n$, $n\ge 3$, be a map  and
	$\pgot\colon \Hcal_1(M;\z)\to\r^n$ be a group morphism.
	Then, there exists a complete conformal minimal immersion $X\colon M\to\r^n$ such that
	\begin{enumerate}[\rm (I)]
		\item $X$ and $Y$ agrees at the points of $\Lambda$.
		\item $\Flux_X=\pgot$.
		\item The generalized Gauss map $G_X\colon M\to\cp^{n-1}$ of $X$ is nondegenerate and fails to intersect $n$ hyperplanes of $\cp^{n-1}$ in general position.
	\end{enumerate}
\end{theorem}
\begin{proof}
	We adapt the argument in \cite{AlarconFernandezLopez2013CVPDE} to prove the theorem.
	We distinguish cases depending on $n\in\n$.
	
	\noindent {\sl Assume $n\in\n$ is even:}\\
	Let $\theta$ be a holomorphic $1$-form never vanishing on $M$. Consider $K\subset M\setminus \Lambda$ a compact disk. Write $\pgot=((\pgot_{j,1},\pgot_{j,2})_{j=1,\ldots, n/2})$ and $Y=((Y_{j,1},Y_{j,2})_{j=1,\ldots, n/2})$.
	Fix $p_0\in K$ and for any $j=1,\ldots,n/2$ take a complex number $\zeta_j\in\c\setminus\{0\}$ and a holomorphic map $f_j=(f_{j,1},f_{j,2})\colon K\to\c^2$ such that
	\begin{enumerate}[\rm (a1)]
		\item[\rm (a1)] $\displaystyle\sum_{j=1}^{n/2}\zeta_j^2=0$.
		\item[\rm (a2)] 
		$f_{j,1}^2+f_{j,2}^2=\zeta_j^2$ on $K$ for any $j=1,\ldots,n/2$.
		\medskip
		\item[\rm (a3)] The map $\big((f_j)_{j=1,\ldots,n/2} \big)$ is nondegenerate.
	\end{enumerate}
	Take into account Lemma \ref{lem:linearindependent} for {\rm (a3)}.
	
	 Next, Proposition \ref{pro:moving2} applied with only interpolation, see Remark \ref{rem:nojet}, for any $j=1,\ldots,n/2$,  to the data 
	\[
	K\subset M,\quad \Lambda,\quad \theta,\quad  H=\zeta_j \quad F=(Y_{j,1},Y_{j,2}),
	\]
	\[
	\quad f_j:=(f_{j,1},f_{j,2}), \quad p_0,\quad \text{and}\quad  \pgot=(\pgot_{j,1},\pgot_{j,2})
	\]
	ensures that $f_j:=(f_{j,1},f_{j,2})$ may be  approximated uniformly on $K$ by  holomorphic maps $\wt f_j=(\wt f_{j,1},\wt f_{j,2})\colon M\to\c^2$ for any $j=1,\ldots,n/2$ such that, defining $\wt f:=((\wt f_{j,1},\wt f_{j,2})_{j=1,\ldots,n/2})\colon M\to\c^n$, we have that
	\begin{enumerate}[\rm (b1)]
		\item $\wt f_{j,1}^2+\wt f_{j,2}^2=\zeta_j^2\neq 0$ on $M$, and so, $\sum_{j=1}^{n/2} \big(\wt f_{j,1}^2+\wt f_{j,2}^2\big)=0$, see {\rm (a1)}.
		\medskip
		\item $\Re(\wt f\theta)$ is an exact $1$-form on $M$. 
		\medskip
		\item $\displaystyle\int_{\gamma} \wt f\theta =\imath\pgot(\gamma)$ for each closed curve $\gamma$ in $M$.
		\item $\displaystyle \Re\int_{p_0}^{p} \wt f\,\theta=Y(p)$ for each $p\in\Lambda$.
		\item $\|\wt f\|^2 \|\theta\|^2 $ is a complete Riemannian metric on $M$.
	\end{enumerate}
	Property {\rm (b5)} follows easily from {\rm (vi)} in Proposition \ref{pro:moving2}. Note that from {\rm (b1)} we have that
	\[||\wt f(p)|| \ge |\wt f_{1,1}(p)|^2+|\wt f_{1,2}(p)|^2 \ge \frac{|\wt f_{1,1}(p)|^2+|\wt f_{1,2}(p)|^2+|\zeta_1|^2}{2}.
	\]
	Therefore, the conformal minimal immersion $X\colon M\to\r^n$ defined by
	\begin{equation}\label{eq:solX}
	X(p):=\Re\int_{p_0}^p  \wt f\,\theta,\quad p\in M
	\end{equation}
	solves the theorem. Indeed, 
	$X$ is a well defined map by {\rm (b2)} and is a conformal minimal immersion by {\rm (b1)}.
	$X$ verifies {\rm (I)} by condition {\rm (b4)} and \eqref{eq:solX}.
	By property {\rm (b5)} we have that $X$ is complete.
	Clearly, {\rm (II)} is implied by {\rm (b3)}, see \eqref{eq:solX}.
	Additionally, if the approximation involved is close enough and taking into account {\rm (a3)}, then 
	$G_X$ is nondegenerate.
	Finally, 
	{\rm (b1)} implies that the generalized Gauss map $G_X$ (see \eqref{eq:ggm}) fails to intersect the hyperplanes given by 
	\begin{equation}\label{eq:planes}
	\bigg\lbrace [(z_1,\ldots,z_n)]\in\cp^{n-1} : z_{2j-1}+(-1)^\delta \imath z_{2j}=0 \bigg\rbrace,
	\end{equation}
	for any  $j=1,\ldots,n/2$ and each $\delta=0,1$,  which are located in general position.
	
	{\noindent \sl Assume now that $n$ is odd.}
	
	As in the previous case, take $K\subset M\setminus \Lambda$ a compact disk and fix $p_0\in K$. Write also $\pgot=(\pgot_1,\ldots,\pgot_n)$ and $Y=(Y_1,\ldots,Y_n)$. By Lemma \ref{lem:specialtheta} there exists a holomorphic $1$-form  $\theta$ never vanishing on $M$ such that
	\begin{equation}\label{eq:gn}
	\displaystyle\int_{\gamma}\theta =\imath \pgot_n(\gamma), \text{ for any loop $\gamma\subset M$},
	\end{equation}
	and also
	\begin{equation}\label{eq:gn2}
	\displaystyle\Re\int_{p_0}^p \theta =Y_n(p), \text{ for any point $p\in\Lambda$}.
	\end{equation}
	Notice that, since $\theta$ is an exact real $1$-form by \eqref{eq:gn}, the values in \eqref{eq:gn2} are independent of the choice of the curve connecting $p_0$ with $p\in\Lambda$.
	
	Take for any $j=1,\ldots,(n-1)/2$ a complex number $\zeta_j\in\c\setminus\{0\}$ and a holomorphic map $f_j=(f_{j,1},f_{j,2})\colon K\to\c^2$ such that
	\begin{itemize}
		\item $\displaystyle\sum_{j=1}^{n/2}\zeta_j^2=-1$.
		\item 
		$f_{j,1}^2+f_{j,2}^2=\zeta_j^2$ on $K$ for any $j=1,\ldots,n/2$.
		\smallskip
		\item The map $\big((f_j)_{j=1,\ldots,n/2} \big)$ is nondegenerate.
	\end{itemize}
	Proposition \ref{pro:moving2} ensures that $f_j:=(f_{j,1},f_{j,2})\colon K\to\c^2$ may be  approximated uniformly on $K$ by holomorphic maps $(\wt f_{j,1},\wt f_{j,2})\colon M\to\c^2$ for any $j=1,\ldots,(n-1)/2$.
	Set 
	\[\wt f:=(\wt f_{1,1},\wt f_{1,2},\ldots,\wt f_{\frac n2,1},\wt f_{\frac n2,2} ,1)\colon M\to\c^n.\]
	
	Reasoning as in the previous case, the conformal minimal immersion $X\colon M\to\r^n$ defined by
	\begin{equation*}
		X(p):=\Re\int_{p_0}^p \wt f\,\theta,\quad p\in M
	\end{equation*}
	solves the theorem,
	provided that the approximation involved is close enough. Here the generalized Gauss map $G_X$ of $X$ fails to intersect the hyperplanes given in \eqref{eq:planes} for $j=1,\ldots, (n-1)/2$ and each $\delta=0,1$, and also the hyperplane given by
\begin{equation}\label{eq:plane}
\bigg\lbrace [(z_1,\ldots,z_n)]\in\cp^{n-1} : z_n=0\bigg\rbrace,
\end{equation}
which are located in general position. 
\end{proof}

\subsection*{Acknowledgements}
This research is a result of the activity developed within the framework of the Programme in Support of Excellence Groups of the Regi\'on de Murcia, Spain, by Fundaci\'on S\'eneca, Science and Technology Agency of the Regi\'on de Murcia.
The author is partially supported by the MINECO/FEDER grant no. MTM2017-89677-P, by MICINN/FEDER project PGC2018-097046-B-I00, Fundaci\'on S\'eneca project 19901/GERM/15, Spain and grant FJC2018-035533-I.

The author wish to express his gratitude to Antonio Alarc\'on for helpful discussions which led to improvement of the paper and to Isabel Fern\'andez for the question which led to these results.



\bigskip

\noindent Ildefonso Castro-Infantes

\noindent Departamento de Matem\'aticas, Universidad de Murcia, Campus de Espinardo, 30100 Murcia, Spain.

\noindent  e-mail: {\tt ildefonso.castro@um.es}

\end{document}